\documentclass{amsart}

\usepackage{amsmath,amssymb,amsfonts,enumerate,amsthm, amscd}

\newcommand{\Hom}{\mbox{Hom}\,}
\newcommand{\Ext}{\mbox{Ext}\,}
\newcommand{\Tor}{\mbox{Tor}\,}

\newcommand{\Ker}{\mbox{Ker}\,}

\renewcommand{\Im}{\mbox{Im}\,}

\newcommand{\pd}{\mbox{pd}\,}
\newcommand{\id}{\mbox{id}\,}
\newcommand{\fd}{\mbox{fd}\,}
\newcommand{\gd}{\mbox{Gdim}\,}
\newcommand{\wdim}{\mbox{wdim}\,}
\newcommand{\gldim}{\mbox{gldim}\,}
\newcommand{\gwd}{\mbox{Gwdim}\,}
\newcommand{\ggldim}{\mbox{Ggldim}\,}
\newcommand{\gpd}{\mbox{Gpd}\,}
\newcommand{\gid}{\mbox{Gid}\,}
\newcommand{\gfd}{\mbox{Gfd}\,}
\newcommand{\GPD}{\mbox{GPD}\,}
\newcommand{\fGPD}{\mbox{fGPD}\,}
\newcommand{\fPD}{\mbox{fPD}\,}
\newcommand{\GID}{\mbox{GID}\,}
\newcommand{\FPid}{\mbox{FP-id}\,}

\newcommand{\fm}{\mathfrak{m}}

\newtheorem{theorem}{Theorem}[section]
\newtheorem{lemma}[theorem]{Lemma}
\newtheorem{proposition}[theorem]{Proposition}

\theoremstyle{definition}
\newtheorem{definition}[theorem]{Definition}
\theoremstyle{remark}
\newtheorem{remark}[theorem]{Remark}

\theoremstyle{Definition and Notation}

\begin{document}
\bibliographystyle{amsplain}

\title[Gorenstein weak dimension ...]{Gorenstein weak dimension of a coherent power series rings}

\author{Najib Mahdou}
\address{Najib Mahdou\\Department of Mathematics, Faculty of Science and Technology of Fez, Box 2202, University S.M. Ben Abdellah Fez, Morocco.}

\author{Mohammed Tamekkante}
\address{Mohammed Tamekkante\\Department of Mathematics, Faculty of Science and Technology of Fez, Box 2202, University S.M. Ben Abdellah Fez, Morocco.}

\author{Siamak Yassemi}
\address{Siamak Yassemi\\Department of Mathematics,
University of Tehran, Tehran, Iran\\ and School of Mathematics,
Institute for research in fundamental sciences (IPM), P.~O.~Box
19395-5746, Tehran, Iran.}




\begin{abstract}

We compute the Gorenstein weak dimension of a coherent power series
rings over a commutative rings and we show that, in general,
$\gwd(R) \leq 1$ does not imply that $R$ is an arithmetical ring.

\end{abstract}

\maketitle

\section{Introduction}

Throughout this paper, all rings are commutative with identity
element, and all modules are unital.

Let $R$ be a ring, and let $M$ be an $R$-module. As usual we use
$\pd_R(M)$, $\id_R(M)$ and $\fd_R(M)$ to denote, respectively, the
classical projective dimension, injective dimension and flat
dimension of $M$. By $\gldim(R)$ and $\wdim(R)$ we denote,
respectively, the classical global dimension and weak dimension of
R.

For a two-sided Noetherian ring $R$, Auslander and Bridger \cite{Aus
bri} introduced the $G$-dimension, $\gd_R (M)$, for every finitely
generated $R$-module $M$. They showed that there is an inequality
$\gd_R (M)\leq \pd_R (M)$ for all finite $R$-modules $M$, and
equality holds if $\pd_R (M)$ is finite.

Several decades later, Enochs and Jenda \cite{Enochs,Enochs2}
defined the notion of Gorenstein projective dimension
($G$-projective dimension for short), as an extension of
$G$-dimension to modules that are not necessarily finitely
generated, and the Gorenstein injective dimension ($G$-injective
dimension for short) as a dual notion of Gorenstein projective
dimension. Then, to complete the analogy with the classical
homological dimension, Enochs, Jenda and Torrecillas \cite{Eno Jenda
Torrecillas} introduced the Gorenstein flat dimension. Some
references are
 \cite{Bennis and Mahdou1, Christensen, Christensen
and Frankild, Enochs2, Eno Jenda Torrecillas, Holm}.

Recently in \cite{Bennis and Mahdou2}, the authors started the study
of the notions global Gorenstein dimensions of ring $R$, which are
denoted by $\GPD(R)$, $\GID(R)$, and $\gwd(R)$ and defined as
follows:

$\begin{array}{cccc}
  (1)& \GPD(R)& =\, \sup\{ \gpd_R(M)\mid M\, \mbox{be an $R$-module}\} \\
  (2)& \GID(R)& =\,  \sup\{ \gid_R(M)\mid M\,\mbox{be an $R$-module}\} \\
  (3)& \gwd(R)& =\, \sup\{ \gfd_R(M)\mid M\,\mbox{be an $R$-module}\}
\end{array}$

They proved that, for any ring R, $ \gwd(R)\leq \GID(R) = \GPD(R)$
(\cite[Theorems 2.1 and 2.11]{Bennis and Mahdou2}). So, according to
the terminology of the classical theory of homological dimensions of
rings, the common value of $\GPD(R)$ and $\GID(R)$ is called
Gorenstein global dimension of $R$, and denoted by $\ggldim(R)$.
They also proved that the Gorenstein global and weak dimensions are
refinement of the classical global and weak dimensions of rings.
That means $\ggldim(R) \leq\gldim(R)$ (resp. $\gwd(R)\leq
\wdim(R)$), and equality holds if $\wdim(R)$) is finite
(\cite[Propositions 2.12]{Bennis and Mahdou2}).

In \cite{JS} J{\o}ndrup and Small gave a connection between a weak
dimension of a coherent power series ring over a commutative ring
$R$ and the weak dimension of $R$, see also \cite[Theorem 8.1.1]{Glaz}. In the following we recall this result:

\begin{theorem}
Let $R$ be a ring, and let $x$ be an indeterminate over $R$. If
$R[[x]]$ is  a coherent ring, then $\wdim(R[[x]])=\wdim(R)+1$.

\end{theorem}

In this paper, we give an extension of Theorem 1.1 to the Gorenstein
weak dimension.

We know that if $\wdim(R) \leq 1$, then $R$ is an arithmetical ring
(see for instance \cite{BG}). Now it is natural to ask the following
question: "Does $\gwd(R) \leq 1$ imply that $R$ is an arithmetical
ring?" In Theorem \ref{Arithmetqiue}, we give a negative answer to
this question. More precisely, we prove: Let $(R,\fm)$ be a local
quasi-Frobenius ring which is not a field. Then $\gwd(R[[X]])= 1$
but $R[[X]]$ is not an arithmetical ring.


\section{Gorenstein weak dimension}

First we recall the notion of strongly Gorenstein projective module
which is introduced in \cite{Bennis and Mahdou1}.

\begin{definition}
 A module $M$ is said to be strongly Gorenstein projective ($SG$-projective for short), if
there exists an exact sequence  of the form:
$$\mathbf{P}=\ \cdots\rightarrow P\stackrel{f}\rightarrow
P\stackrel{f} \rightarrow P\stackrel{f}
         \rightarrow P \rightarrow\cdots$$ where $P$ is a projective $R$-module and $f$ is an endomorphism of $P$, such
          that  $M \cong \Im(f)$ and such that $\mathbf{Hom}(-,Q)$ leaves the sequence $\mathbf{P}$ exact whenever
           $Q$ is a projective module.
\end{definition}

These strongly Gorenstein projective modules has a simple
characterization, and they are used to characterize the Gorenstein
projective modules. We recall the following two results which are \cite[Propositions 2.9]{Bennis and Mahdou1} and \cite[Theorem 2.7]{Bennis and Mahdou1}:

\begin{proposition}
A module M is strongly Gorenstein projective if, and only if, there
exists a short exact sequence of modules: $$0\longrightarrow M
\longrightarrow P \longrightarrow M \longrightarrow 0$$ where $P$ is
projective and $\Ext(M,Q) = 0$ for any projective module $Q$.

\end{proposition}

\begin{theorem} A  module is Gorenstein projective
if, and only if, it is a direct summand of a strongly Gorenstein
projective  module.
\end{theorem}

\begin{lemma}\label{lemma x is nonzero}
Let $R$ be a ring and let $X$ be an indeterminate over  $R$ and $M$
an $R[[X]]$-module. Then $X$ is a nonzero divisor on $M$ if and
only if $\Tor_{R[[X]]}(M,R)=0$.
\end{lemma}

\begin{proof}  Let $\varphi_X: M\longrightarrow M$ be the homomorphism of $R[[X]]$-module such that $\varphi_X(m)=Xm$ for every $m\in M$. Consider the short exact sequence of $R[[X]]$-modules
$$(\star) \quad 0 \longrightarrow R[[X]] \stackrel{\mu_X}\longrightarrow
R[[X]] \longrightarrow R\cong R[[X]]/XR[[X]] \longrightarrow 0$$
where $\mu_X$ is the multiplication by $X$. The following sequence
is induced from $(\star)$
$$0\longrightarrow \Tor_{R[[X]]}(M,R) \longrightarrow R[[X]]\otimes_{R[[X]]}
M \stackrel{1_M\otimes \mu_X}\longrightarrow R[[X]]\otimes_{R[[X]]}
M \longrightarrow R\otimes_{R[[X]]}M \longrightarrow 0$$ By
\cite[Theorem 8.13]{Rotman} the $R[[X]]$-morphism  $1_M\otimes
\mu_X$ is multiplication by $X$. So, in the following diagram all
squares are commutative

$$\begin{array}{ccccccccc}
  0\rightarrow &\Tor_{R[[X]]}(M,R) & \rightarrow & R[[X]]\otimes_{R[[X]]}M & \stackrel{1_M\otimes \mu_X}\rightarrow& R[[X]]\otimes_{R[[X]]}M & \rightarrow & R\otimes_{R[[X]]}M \rightarrow 0 &  \\
   &  &  & \wr \parallel &  & \wr \parallel &  & \wr \parallel &  \\
 0 \rightarrow & Ker(\varphi_X) & \rightarrow & M & \stackrel{\varphi_X}\rightarrow & M & \rightarrow & M/XM \quad \rightarrow 0&  \\
\end{array}$$

Therefore, $\Ker(\varphi_X)\cong \Tor_{R[[X]]}(M,R)$ and hence $X$
is a nonzero divisor on $M$ if, and only if, $\Tor_{R[[X]]}(M,R)=0$.
\end{proof}

\begin{lemma}\label{thm coherent M/XM}
Let $R$ be a  ring  and $X$ an indeterminate over $R$ such that
$R[[X]]$ is coherent. If $M$ is a finitely presented $R[[X]]$-module
such that $X$ is a nonzero divisor  on $M$ then
$\gpd_{R[[X]]}(M)\leq  \gpd_R(M/XM)$.
\end{lemma}

\begin{proof} First note that $R$ is a coherent ring
 by \cite[Theorem 4.1.1(1)]{Glaz}. In addition, $X$ is contained
in the Jacobson radical of $R[[X]]$. Let $M$ be a finitely presented
$R[[X]]$-module $M$ over which $X$ is a nonzero divisor and put $n=\gpd_R(M/XM)$. We may assume that $n$ is finite.\\
 The proof will be by induction on $n$.\\
 If  $M/XM$ is a
Gorenstein projective $R$-module, then by using \cite[Proposition
10.2.6 (1)$\Leftrightarrow$(10)]{Enochs relative}, the proof  is the
same as the one of \cite[Corollary 1.4.6]{Christensen} (note that $X$ is an element of
 the Jacobson radical of $R[[X]]$ and so we may use the Nakayama's Lemma in the proof of  \cite[Corollary 1.4.6]{Christensen}. In the original proof we use the \emph{Local} condition) .\\
Now, assume that $n>0$ and consider the short exact
 sequence of $R[[X]]$-modules  $0 \longrightarrow G \longrightarrow P \longrightarrow
 M\longrightarrow 0$ where $P$ is a finitely presented  projective $R[[X]]$-module. Using \cite[Theorem 2.5.1]{Glaz},  $G$ is also  finitely presented since $R[[X]]$ is coherent.
 From Lemma \ref{lemma x is nonzero}, we have $\Tor_{R[[X]]}(M,R)=0$ since $X$ is
 a nonzero divisor on $M$. In addition, $\Tor_{R[[X]]}(P,R)=0$ since $P$ is a projective $R[[X]]$-module. Therefore, $\Tor_{R[[X]]}(G,R)=0$ (since $\fd_{R[[X]]}R\leq 1$). So, by Lemma \ref{lemma x is
 nonzero}, $X$ is a nonzero divisor on $G$. On the other hand, if we tensor  the short exact sequence
 above with $-\otimes_{R[[X]]}R$ we obtain a short exact sequence
 $$0\longrightarrow G/XG \longrightarrow P/XP \longrightarrow M/XM
 \longrightarrow0$$ (note that $M\otimes_{R[[X]]}R\cong M/XM$). Therefore, by the hypothesis
 condition of induction, $\gpd_{R[[X]]}(G)\leq \gpd_R(G/XG)\leq
 n-1$. Thus, $\gpd_{R[[X]]}(M)\leq \gpd_{R[[X]]}(G)+1\leq n$, as
 desired.
 \end{proof}

\begin{definition}[\cite{Stenstrom} and \cite{Garkusha}] Let $R$ be a
ring and let $M$ be an $R$-module.
\begin{enumerate}
    \item We say that $M$ has $FP$-injective dimension at most n (for some $n\geq 0$), denoted by $\FPid_R(M)\leq n$, if $\Ext^{n+1}_R(P,M) = 0$ for every finitely
     presented $R$-module $P$.
    \item A ring $R$ is said to be $n-FC$, if it is coherent and it has
    self-$FP$-injective at most at $n$ (i.e., $\FPid_R(R)\leq n$).
\end{enumerate}
\end{definition}

A ring is called $FC$ ring if it is 0-FC.

Using \cite[Theorems 6 and 7]{Chen}, we deduce the following Lemma.

\begin{lemma}\label{n-FC}
Let $R$ be a coherent ring and let $n\ge 0$ be an integer. The
following are equivalent:
\begin{enumerate}
    \item $R$ is $n-FC$;
    \item $\gwd(R)\leq n$;
    \item $\gpd_R(M)\leq n$ for every finitely presented $R$-module
    $M$.
\end{enumerate}
\end{lemma}

\begin{remark}\label{rem}
\begin{enumerate}
\item By Lemma \ref{n-FC}, the Gorenstein weak dimension of a coherent ring $R$
is also determined by the formula:
$$\gwd(R)=\; sup\{\; \gpd_R(M)|\; M \mbox{is a finitely presented $R$-module}\}.$$

\item In Lemma \ref{n-FC}, the case $n=0$ (i.e., if $R$ is $FC$) does not
need the coherence condition (see \cite[Theorem 6]{Chen}).
\end{enumerate}
\end{remark}

\begin{lemma}\label{gwd(R)>gwd(R/xR)}
Let $R$ be a coherent ring and let $X$ be an indeterminate over $R$.
Then, $\gwd(R[[X]])\geq \gwd(R) + 1$.
\end{lemma}

\begin{proof} By \cite[Theorem 4.1.1(1)]{Glaz},
$R\cong R[[X]]/XR[[X]]$ is coherent since it  is a finitely
presented $R$-module (from the short exact sequence\\ $0\longrightarrow R[[X]]\stackrel{X}\longrightarrow R[[X]] \longrightarrow R \longrightarrow 0$).\\
We may assume that $\gwd(R[[X]])=n<\infty$. Using \cite[Theorem
1.3.3]{Glaz} and \cite[Proposition 2.27]{Holm}, we have
$\gpd_{R[[X]]}(R)=\pd_{R[[X]]}(R)=1$. Thus, by Lemma \ref{n-FC},
$\gwd(R[[X]])=n\geq 1$ since $R$ is a finitely
presented $R[[X]]$-module.\\
 Now, let $M$ be a finitely presented $R$-module. Then, by
 \cite[Theorem 2.1.8]{Glaz}, $M$ is a finitely presented
 $R[[X]]$-module (since $R\cong R[[X]]/XR[[X]]$). Thus, by \cite[Theorem 1.3.5]{Glaz} and Lemma \ref{n-FC},
 $\Ext^n_{R}(M,R)=\Ext^{n+1}_{R[[X]]}(M,R[[[X]])=0$. Therefore, $R$ is $(n-1)-FC$. Hence, by Lemma \ref{n-FC},  $\gwd(R)\leq n-1$. Therefore $\gwd(R)\leq \gwd(R[[X]])-1$, as
 desired.
 \end{proof}

Now we are ready to present our main result of this paper.

\begin{theorem}\label{main result}
Let $R$ be a ring and let $x$ be an indeterminate over $R$. If
$R[[x]]$ is  a coherent ring, then $\gwd(R[[x]])=\gwd(R)+1$.

\end{theorem}

\begin{proof} If $\gwd(R)=\infty$, then by Lemma
\ref{gwd(R)>gwd(R/xR)}, we have the desired equality. Otherwise we
put $\gwd(R)=n$. By Lemma \ref{gwd(R)>gwd(R/xR)}, it is enough to
show that $\gwd(R[[X]])\leq \gwd(R)+1$. Let $M$ be a finitely
presented
 $R[[X]]$-module and consider a short exact sequence of
 $R[[X]]$-modules $$(*)\,\,\,\,0\longrightarrow K \longrightarrow P
 \longrightarrow M\longrightarrow 0,$$ where $P$ is a finitely
 generated projective $R[[X]]$-module. Then, by \cite[Theorem 2.5.1]{Glaz}, $K$ is also finitely presented
 since $R[[X]]$ is coherent. Thus $K/XK$ is also finitely presented $R$-module (by \cite[Theorem 2.1.8]{Glaz}). On the other hand, from the short sequence $(*)$ we
 have $\Tor_{R[[X]]}(K,R)=\Tor_{R[[X]]}^2(M,R)=0$ since $\fd_{R[[X]]}(R)\leq
 1$. So, from Lemma \ref{lemma x is nonzero}, $X$ is a nonzero
 divisor on $K$. Then, by Lemma \ref{thm coherent M/XM} and Lemma \ref{n-FC},
 $\gpd_{R[[X]]}(K)\leq \gpd_R(K/XK)\leq n$. Then, $\gpd_{R[[X]]}(M)\leq
 n+1$. Consequently, by Lemma \ref{n-FC}, $\gwd(R[[X]])\leq n+1
 =\gwd(R)+1$.
\end{proof}

Recall that a ring is called quasi-Frobenius, if it is Noetherian
and  self-injective (see \cite{Nicholson}).

\begin{proposition}\label{QF}
Let $R$ be a quasi-Frobenius ring. Then,
$$\gwd(R[[X]])=\ggldim(R[[X]])=1$$
\end{proposition}

\begin{proof} Let $R$ be a quasi-Frobenius ring. Then, from
\cite[Proposition 2.8 and Theorem 2.9]{Bennis and Mahdou2},
$\gwd(R)=\ggldim(R)=0$. Thus, from Theorem \ref{main result},
$\gwd(R[[X]])=1$. On the other hand, $R$ is Noetherian and so
$R[[X]]$ is also Noetherian. Therefore, by \cite[Theorem 2.9]{Bennis
and Mahdou2}, $\ggldim(R[[X]])=1$.
\end{proof}

Recall that a ring $R$ is called an arithmetical ring if every
finitely generated ideal is locally principal. If $\wdim(R) \leq 1$,
then $R$ is an arithmetical ring (see for instance \cite{BG}). So we
lead to ask the following question: If $\gwd(R) \leq 1$, then is $R$
arithmetical ring?

The following result shows that the above question is false in
general.

\begin{theorem}\label{Arithmetqiue}
Let $(R,\fm)$ be a local quasi-Frobenius ring which is not a field.
Then the following statements hold:
\begin{itemize}
    \item[(1)] $\gwd(R[[X]])= 1$.
    \item[(2)] $R[[X]]$ is not an arithmetical ring.
\end{itemize}
\end{theorem}

\begin{proof} (1) We have $\gwd(R[[X]])= 1$ by Proposition \ref{QF}.

{\bf 2)} We claim that $R[[X]]$ is not an arithmetical ring. Deny. Let $a$
be a non-zero non-invertible element of $R$ and let $I :=aR[[X]] +
XR[[X]]$. Then $I =PR[[X]]$ for some $P :=\sum _{i}a_i X^i\in
R[[X]]$ (where $a_i \in R$), since $R[[X]]$ is a local arithmetical
ring.

Since $P \in I =aR[[X]] + XR[[X]]$, we have $P =aQ_{1} + XQ_{2}$ for
some $Q_{1} (:=\sum _{i}c_i X^i$), $Q_{2} \in R[[X]]$ . Hence,
$a_{0} =ac_{0}$.

On the other hand, we have $a =PQ$ for some $Q =\displaystyle\sum
_{i}b_i X^i \in R[[X]]$ (where $b_{i} \in R$) since $a \in I
=PR[[X]]$. Hence, $a =a_{0}b_{0}$. We claim that $b_{0} \in M$. If
this is not the case, then $b_0$ is invertible in $R$ and so $Q$ is
invertible in $R[[X]]$; hence, we may assume that $P =a$ (since
$aR[[X]] =PQR[[X]] =PR[[X]] =I$). But, $X \in I =aR[[X]]$ implies
that $X =a\sum _{i}d_i X^i$ for some $d_{i} \in R$. Hence, $1
=ad_{1}$ and so $a$ is invertible in $R$, a contradiction.
Therefore, $b_{0} \in M$.

Therefore, $a =ab_{0}c_{0}$ since $a_{0} =ac_{0}$ and $a
=a_{0}b_{0}$  and so $a(1 - b_{0}c_{0}) =0$. But $1 - b_{0}c_{0}$ is
invertible in $R$ since $b_{0}c_{0} \in M$ (since $b_{0} \in M$);
hence $a =0$, a contradiction. Hence, $R[[X]]$ is not an
arithmetical ring, as desired.
\end{proof}

In the rest of this paper, We compare the small finitistic
Gorenstein projective dimension of the base ring R,
$$\fGPD(R)=\sup \left\{%
\begin{array}{ll}
    &| \quad M \; is \; an \; R-module\; with\  \\
    \gpd_R(M) & |\quad  finite\; Gorenstein \; projective \; dimension\\
     & |\quad and\; M\; admits \; a\; finite \; projective \; resolution \\
\end{array}
\right\}.$$ with the usual small finitistic projective dimension,
$\fPD(R)$ (see \cite{Glaz}).

It is clear that if $R$ is coherent we have
 $$\fGPD(R)= sup\{\gpd_RM \, | \, M \mbox{is a finitely presented and}\,\,
 \gpd_RM<\infty\}\leq \gwd(R),$$
 with equality  if $\gwd(R)$ is finite (by Lemma \ref{n-FC}).

In the proof of the next Theorem, we use the proofs of
\cite[Theorems 2.10 and  2.28]{Holm}.

\begin{theorem}\label{fGPD}
For any coherent ring $R$ there is an equality $\fGPD(R)=\fPD(R)$.
\end{theorem}

Recall that a right co-proper projective resolution of an $R$-module
$M$ is an exact sequence $\mathbf{X}= 0\longrightarrow M
\longrightarrow X^0\longrightarrow X^1\longrightarrow ...$ with
$X^i$ is projective for each $i\geq 0$ such that
$\mathbf{Hom}_R(\mathbf{X},P)$ is exact for every projective module
$P$.

\begin{lemma}\label{lemma resolution}
Let $R$ be a ring. Then, every finitely generated $G$-projective
$R$-module admits a right co-proper resolution of finitely generated
free $R$-module.
\end{lemma}

\begin{proof} Let $M$ be a finitely generated  $G$-projective $R$-module.
By \cite[Proposition 2.4]{Holm}, there is an exact sequence of
$R$-modules $$0 \longrightarrow M \longrightarrow L \longrightarrow
M' \longrightarrow 0$$ where $L$ is a free $R$-module and $M'$ a
$G$-projective $R$-module.  We identify $M$ to a submodule of $L$
and we assume that $L$ admits basis $\{x_k,\; k \in K\}$. Since $M$
is finitely generated and each generator of $M$ is a finite linear
combination of finite subset of $\{x_k,\; k \in K\}$, we consider
$L_0$ a finitely generated free direct summand of $L$ which contains
$M$ and a free $R$-module $L_1$ such that $L=L_0\oplus L_1$. Then,
we have the following commutative diagram
$$(1)\qquad  \begin{array}{ccccccc}
   &  &  & 0 &  &  &  \\
   &  &  & \downarrow &  &  &  \\
  0\longrightarrow & M & \hookrightarrow & L_0 & \longrightarrow & L_0/M & \longrightarrow 0 \\
   & \parallel &  & \downarrow &  &  &  \\
  0\longrightarrow & M & \hookrightarrow & L & \longrightarrow & M'& \longrightarrow 0 \\
   &  &  & \downarrow &  & &  \\
   &  &  &  L_1 & &   & \\
   &  &  & \downarrow &  &  &  \\
   &  &  & 0 &  &  &  \\
\end{array}$$
 From \cite[Exercise 2.7 page 29]{Rotman} the diagram above can be completed as

$$(2)\qquad \begin{array}{ccccccc}
   &  &  & 0 &  & 0 &  \\
   &  &  & \downarrow &  & \downarrow &  \\
  0\longrightarrow & M & \longrightarrow & L_0 & \longrightarrow & L_0/M & \longrightarrow 0 \\
   & \parallel &  & \downarrow &  & \downarrow &  \\
  0\longrightarrow & M & \longrightarrow & L & \longrightarrow & M'& \longrightarrow 0 \\
   &  &  & \downarrow &  & \downarrow&  \\
   &  &  &  L_1 &= &  L_1 & \\
   &  &  & \downarrow &  & \downarrow &  \\
   &  &  & 0 &  & 0 &  \\
\end{array}$$

 In the right vertical exact sequence $L_1$ is a projective
 module and so $M'\cong L_0/M\oplus L_1$. Then, from
 \cite[Theorem 2.5]{Holm}, $L_0/M$ is a $G$-projective $R$-module. In addition, by \cite[Theorem 2.20]{Holm}, for every projective $R$-module $F$, the short sequence
 $$0\longrightarrow \Hom_R(L_0/M,F) \longrightarrow \Hom_R(L_0,F) \longrightarrow \Hom_R(M,F) \longrightarrow 0$$ is exact (since $\Ext(L_0/M,F)=0$). On the other hand, it is
 clear that $L_0/M$ is  finitely generated. Then, by repeating this procedure, $M$ admits a right co-proper resolution of finitely generated
 free $R$-modules.
 \end{proof}
\begin{proof}[Proof of Theorem \ref{fGPD}] Clearly $\fPD(R)\leq \fGPD(R)$ by \cite[Proposition
2.27]{Holm}. In first we claim that $\fGPD(R)\leq \fPD(R)+1$. So,
let $M$ be  a finitely presented module with $0<\gpd_R(M)=n<\infty$.
We may pick an exact sequence, $0\longrightarrow K'\longrightarrow
P_{n-1}\longrightarrow ...\longrightarrow P_0\longrightarrow M
\longrightarrow 0$, where $P_0,...,P_{n-1}$ are finitely generated
projective and $K'$ is finitely generated $G$-projective (since $R$
is coherent and by \cite[Proposition 2.7]{Holm}). On the other hand,
by Lemma \ref{lemma resolution}, there is an exact sequence
$0\longrightarrow K' \longrightarrow L^0\longrightarrow
...\longrightarrow L^{n-1}\longrightarrow G \longrightarrow 0$ where
$L^0,...,L^{n-1}$ are finitely generated free module and $G$ is
finitely generated $G$-projective modules and such that
$\Hom_R(-,Q)$ leaves this  sequence exact, whenever $Q$ is
projective. Thus, by \cite[Proposition 1.8]{Holm} there exist
homomorphisms, $L^i\rightarrow P_{n-1}$ for $i=0,...,n-1$, and
$G\rightarrow M$, such that the following diagram is commutative.
$$\begin{array}{ccccccccccc}
  0\longrightarrow & K' & \longrightarrow & L^0 & \longrightarrow & ... & \longrightarrow & L^{n-1} & \longrightarrow & G  & \longrightarrow 0 \\
   & \| &  & \downarrow  &  &  &  & \downarrow &  & \downarrow &  \\
 0\longrightarrow & K' & \longrightarrow & P_{n-1} & \longrightarrow & ... & \longrightarrow & P_0 & \longrightarrow & M  & \longrightarrow 0 \\
\end{array}$$

 The diagram gives a chain map between complexes,

$$\begin{array}{cccccccccc}
  0\longrightarrow  & L^0 & \longrightarrow & ... & \longrightarrow & L^{n-1} & \longrightarrow & G  & \longrightarrow 0 \\
   & \downarrow    &  &  &  & \downarrow &  &\downarrow &  \\
 0\longrightarrow   & P_{n-1} & \longrightarrow & ... & \longrightarrow & P_0 & \longrightarrow & M & \longrightarrow 0 \\
\end{array}$$
which induces an isomorphism in homology. Its mapping cone is exact
and all the modules in it, except for $P_0\oplus G$ (which is
finitely generated Gorenstein projective), are finitely presented
projective. Hence the kernel $K$ of $\varphi: P_0\oplus
G\twoheadrightarrow M$ satisfies $\pd_RK\leq n-1$ (and then
necessarily $\pd_RK=n-1$ by \cite[Propositions 2.18 and 2.27]{Holm})
and it is finitely presented since all $P_i$ and $L^i$ are finitely
generated. Thus, we get $\fGPD(R)\leq \fPD(R)+1$. Proving the
inequality $\fGPD(R)\leq
\fPD(R)$, we may therefore assume that $0< \fGPD(R)=m<\infty$.\\
Pick a finitely presented module $M$ with $\gpd_RM=m$. We wish to
find a finitely presented module $Q$ with $\pd_RQ=m$. By the above proof
there is an exact sequence $$0\longrightarrow K
\longrightarrow G \longrightarrow M \longrightarrow 0$$ where $G$ is
a finitely generated Gorenstein projective module  and $K$ is a
finitely presented module such that $\pd_RK=m-1$. Since $G$ is a
finitely generated Gorenstein projective module, there exists, by
Lemma \ref{lemma resolution}, a finitely generated free module $L$
with $G \subseteq L$, and since also  $K \subseteq G$, we can
consider the quotient $Q=L/K$ ($Q$ is finitely presented by
\cite[Theorem 2.5.1]{Glaz}). Note that $M\cong G/K$ is a submodule
of $Q$, and that we get a short exact sequence
$$0 \longrightarrow M \longrightarrow Q \longrightarrow
Q/M\longrightarrow 0.$$ If $Q$ is Gorenstein projective,
\cite[Proposition 2.18]{Holm} implies that $\gpd_R(Q/M)=m+1$, since
$\gpd_RM=m$. But this contradict the fact that $m=\fGPD(R)<\infty$
since $Q/M$ is finitely presented (by \cite[Theorem 2.5.1]{Glaz}
since $Q$ and $M$ are finitely presented and $R$ is coherent). Hence
$Q$ is not Gorenstein projective, in particular, $Q$ is not
projective. Therefore the short exact sequence $0 \longrightarrow K
\longrightarrow L \longrightarrow Q \longrightarrow 0$ shows that
$\pd_RQ=\pd_RK+1=m$.

\end{proof}

\begin{proposition} let $R$ be a coherent ring and let $x$ be a nonzero
divisor in $R$ contained in the intersection of the maximal ideals
of $R$. Then: \begin{enumerate}
    \item $\fGPD(R)=\fGPD(R/xR)+1$
     \item If $\gwd(R)<\infty$, then $\gwd(R)=\gwd(R/xR)+1$
\end{enumerate}
\end{proposition}

\begin{proof} The first equality follows from Theorem \ref{fGPD} and
\cite[Corollary 3.1.4]{Glaz}. Now assume that $\gwd(R)=n$ is finite.
We claim that $\gwd(R/xR)$ is finite. In first see, by \cite[Theorem
4.1.1(1)]{Glaz}, that $R/xR$ is also coherent since $R/xR$ is a
finitely presented $R$-module (from the short exact sequence
$0\longrightarrow R\stackrel{x}\longrightarrow R \longrightarrow
R/xR \longrightarrow 0$). Using \cite[Theorem 1.3.3]{Glaz} and
\cite[Proposition 2.27]{Holm}, we have
$$\gpd_{R}(R/xR)=\pd_{R}(R/xR)=1$$ Then, from Lemma \ref{n-FC},
$\gwd(R)=n\geq 1$ since $R/xR$ is a finitely presented $R$-module.
 Now, let $M$ be a finitely presented $R/xR$-module. Then, by
 \cite[Theorem 2.1.8]{Glaz}, $M$ is a finitely presented
 $R$-module. Thus, by \cite[Theorem 1.3.5]{Glaz} and Lemma \ref{n-FC},
 $\Ext^n_R(M,R/xR)=\Ext^{n+1}_{R}(M,R)=0$. Therefore, $R/xR$ is $(n-1)-FC$. Hence, by Lemma \ref{n-FC},  $\gwd(R/xR)\leq n-1<\infty$.
 So, by Theorem \ref{fGPD} and
\cite[Corollary 3.1.4]{Glaz}, we have
$$\gwd(R)=\fGPD(R)=\fGPD(R/xR)+1=\gwd(R/xR)+1.$$ Now the assertion holds.
\end{proof}




\bigskip\bigskip


\end{document}